\documentclass[english,reqno]{amsart}
\usepackage[margin=1in]{geometry}

\usepackage{preamble}

\begin{document}

\setlength{\abovedisplayskip}{10pt}
\setlength{\belowdisplayskip}{10pt}

\title{Uniqueness of Solutions to a Gas-Disk Interaction System}

\author{Anton Iatcenko}

\address{Department of Mathematics, Simon Fraser University, 8888 University Dr., Burnaby, BC V5A 1S6, Canada}
\email{anton\_iatcenko@sfu.ca}

\author{Weiran Sun}
\address{Department of Mathematics, Simon Fraser University, 8888 University Dr., Burnaby, BC V5A 1S6, Canada}
\email{weiran\_sun@sfu.ca}

%\date{\today}
\date{}
\begin{abstract}
  In this paper we give an elementary proof of uniqueness of solutions to a gas-disk interaction system with diffusive
  boundary condition. Existence of near-equilibrium solutions for this type of systems with various boundary conditions has
  been extensively studied in \cite{ACMP2008, CS2014, CS2015, CMP2006, CCM2007, SR2014, K2018, C2007, K2018-2}. However, the
  uniqueness has been an open problem, even for solutions near equilibrium. Our work gives the first rigorous proof of the
  uniqueness among solutions that are only required to be locally Lipschitz; in particular, it holds for solutions far from
  equilibrium states. \vspace{2mm}
  
  \noindent
  \hspace{-5mm}\textbf{Keywords:} kinetic equations, integro-differential equations, uniqueness, gas-body
   interaction, friction. 
\end{abstract}

\maketitle

\section{Introduction}
The main goal of this paper is to show uniqueness of solutions to a gas-disk interaction system. This system describes the motion of a disk immersed in a collisionless gas. Among many ways to model the friction between the gas and the disk, the simplest one is to assume that the friction is proportional to the velocity of the disk. In this scenario the velocity of the disk can be found by solving a linear ODE. Here we consider a more realistic model as in \cite{ACMP2008, CS2014, CS2015, CMP2006, CCM2007, SR2014, K2018, C2007, K2018-2, TA2012, TA2013, TA2014}, where the evolution of the gas and the disk satisfies a coupled system of integro-differential equations. The coupling is through collisions of gas particles with the disk: these collisions produce a drag force on the disk through momentum exchange and provide a boundary condition for the gas. 

In this paper we make a simplifying assumption that the disk is infinite. Together with assumed symmetry this lets us reduce
the whole system to one dimension, thus making the disk a single point moving along the horizontal axis. To specify the model we
let $f(x, v, t)$ be the density function of the gas that evolves according to the free transport equation away from the disk: 
\begin{align} \label{eq:gas}
\del_t f + v \, \del_x f = 0 ,  \qquad   f(x, v, 0) = \phi_0(v) ,
\end{align}
where $(x, v, t) \in \R \times \R \times \R^+$ are position, velocity, and time respectively. Denote the position of the
disk at time $t$ by $\eta(t)$ and its velocity by $p(t)$. The interaction of the gas with the disk is described by a diffusive
boundary condition: %\vspace{-5mm}
\begin{align}
f^+_R(\eta(t), v, t) &= 2 e^{-(v - p(t))^2} \int_{-\infty}^{p(t)} (p(t) - w) f^-_R(\eta(t), w, t) \dw, 
\quad v > p(t), \label{BC:R} \\
f^+_L(\eta(t), v, t) &= 2 e^{-(p(t) - v)^2} \int^{ \infty}_{p(t)} \, \ (w - p(t)) f^-_L(\eta(t), w, t) \dw, 
\quad v < p(t), \label{BC:L}
\end{align}
where the sub-indices $R$ and $L$ denote the right and left sides of the disk. Throughout the paper superscripts $+$ and $-$ 
on the density functions denote the postcollisional and precollisional distributions respectively, understood as one-sided
 limits:
\begin{align} \label{def:f_pm}
f^\pm(\eta(t), v, t) = \lim_{\Eps \to 0^+} f^\pm(\Eta(t)\pm \Eps v, v, t \pm \Eps) .
\end{align}

%\newpage
The diffusive boundary conditions~\eqref{BC:R}-\eqref{BC:L} essentially state that shape of the outgoing distribution is always
Gaussian, with coefficients chosen to ensure the conservation of mass. Therefore, our model considers the case where collisions
are instantaneous and the disk does not capture any finite mass of the gas via the collision process.

We assume that the disk is acted on by an external force $F(x, t)$ and the drag force $G_p(t)$ generated through collisions with
the gas particles (we have associated the drag force with the sub-index $p$ to emphasize its dependence on the disk velocity
$p$).  Then the motion of the disk is described by
\begin{alignat}{2} 
\dot    p &= F(\eta(t), t) - G_p(t),  \qquad &&p(0) = p_0,  \label{eq:disk} \\
\dot \eta &= p(t),                        &&\eta(0) = 0. \label{eq:disk-eta}
\end{alignat}
We will write the total drag force as a combination of the drag forces due to particles colliding with
the disk from the right and left: \vspace{-5mm}
\begin{align} \label{def:drag-full}
G_p(t) = G_{p, R}(t) - G_{p, L}(t).
\end{align}
The signs are chosen to make both components of the drag positive. Physically speaking, $G_{p, L}$ accelerates the disk
and $G_{p, R}$ decelerates it. Their exact expressions are derived from Newton's Second Law (see \cite{CS2014} for details):
\vspace{-5mm}
\begin{align} 
G_{p, R}(t) &:= \int^{p(t)}_{-\infty} (p(t) - v)^2 f^-_R(t, \Eta(t), v) \dv + 
\int^\infty_{p(t)} (v - p(t))^2 f^+_R(t, \Eta(t), v) \dv \,, \label{def:drag-R}\\
G_{p, L}(t) &:= \int^{p(t)}_{-\infty} (p(t) - v)^2 f^+_L(t, \Eta(t), v) \dv + 
\int^\infty_{p(t)} (v - p(t))^2 f^-_L(t, \Eta(t), v) \dv \,. \label{def:drag-L}
\end{align}
The evolution of the complete gas-disk system is governed by equations~\eqref{eq:gas}-\eqref{def:drag-L}.
We comment that the derivation of~\eqref{def:drag-R}-\eqref{def:drag-L} relies on the Reynolds transport theorem, 
which assumes that the exchange of momentum between the gas and the disk can only happen through the fluxes of the gas 
moving into and out of the disk. Hence any particle that stays on the disk does not contribute to the momentum exchange 
or the drag force. We also note that to have an interaction with the disk the particle to the right (left) of it must be moving 
slower (faster) than the disk. 

Gas-body coupled systems have been extensively studied both numerically and analytically with pure diffusive, specular, and 
more generally, the Maxwell boundary conditions (\cite{ACMP2008, TA2012, TA2013, TA2014, CS2014, CS2015, CMP2006,
CCM2007, SR2014, K2018, C2007, K2018-2}). We refer the reader to a recent paper \cite{K2018} for a comprehensive list of
references. Among the central questions for these systems are their well-posedness and long-time behaviour. Regarding 
the long-time asymptotics, it is now fairly well-understood that due to the effect of re-collisions, the relaxation of the disks
velocity toward its equilibrium state may not be exponential as in the simplified model where re-collisions are ignored. In
fact, one may obtain algebraic decay rates \cite{ACMP2008, TA2012, TA2013, TA2014, CS2014, CS2015, CMP2006, CCM2007, SR2014,
K2018, C2007, K2018-2}. Moreover, depending on the shape of the body, such rates may or may not depend on the spatial dimension
\cite{C2007,SR2014}. 

The well-posedness issue, however, is less understood. To the best of our knowledge, existence of solutions has only
been investigated for data near equilibrium states \cite{ACMP2008, CS2014, CS2015, CMP2006, CCM2007, SR2014, K2018, C2007,
K2018-2} and uniqueness has been an open question even for these solutions. It is our goal in this paper to give 
a uniqueness proof for solutions to~\eqref{eq:gas}-\eqref{def:drag-L}, where the disk velocity $p$ only needs to be in the
natural space of locally Lipschitz functions. This includes solution spaces considered in~\cite{ACMP2008, CS2014}, as well as 
more general cases with solutions far from an equilibrium. 
%In fact, the systems we are considering may not have any 
%equilibrium state at all since we are allowing for a fairly general external forcing. 
The main result of this paper is

\begin{thm} \label{thm:main}
Suppose the initial density $\phi_0 \in L^1 \cap L^\infty (\R)$ and the external force $F(x, t)$ is Lipschitz in $x$ 
with its Lipschitz coefficient independent of $t$. Then for any $p_0 \in \R$ there exists at most one solution $(\eta, p, f)$ to the system
\eqref{eq:gas}-\eqref{def:drag-L} such that $p$ is locally Lipschitz.
\end{thm}

Our main step in proving the main theorem is to show that the drag force due to recollisions, denoted by $G_\text{rec}$, is Lipschitz in the velocity $p$
(Proposition~\ref{prop:Lip-G}). The main difficulty for such estimate is the dependence the distribution of the recolliding
particles on the entire history of the disk motion. We address this issue by taking advantage of the inherently recursive
nature of the problem: the distribution of the particles colliding with the disk for the $n^\text{th}$ time at time $t$ is
determined by the distribution of the particles colliding with the disk for the $(n-1)^\text{th}$ time at some earlier time $s$.
Instead of trying to compute or estimate such $s$ for a given velocity $v$, we use a change of variables $v = v(s, t)$. This
allows us to compare the particles that have collided with the disk at the same time in the past instead of comparing particles
that have the same velocity at the current time. 

Three remarks are in order: first, we have assumed that the initial state of the gas is spatially homogeneous. This
assumption can be dropped at the cost of adding more technicalities. Second, due to the essential step of
change of variables, so far our technique is only applicable to the case with diffusive boundary conditions. Hence for systems
with specular or Maxwell boundary conditions uniqueness is still an open question. Third, this paper only deals with the 
one-dimensional case with a collisionless gas, but we expect a similar strategy to be applicable in higher dimensions and for
systems with simple collisions such as the special Lorenz gas in \cite{TA2012}. This will be subject to future investigation. 

The rest of the paper is laid out as follows: in Section~\ref{sec:Assump} we state our assumptions, introduce partition of the 
density function and the change of variables, and reformulate the density function and the drag force into recursive forms. 
Section~\ref{sec:Assump} contains the essential ideas and constructions that will be used in various estimates and the uniqueness proof in the later part. In Section~\ref{sec:prelim_bnds} we obtain
preliminary bounds on the density function using the recursive form. Finally, in Section~\ref{sec:uniq} we establish the
Lipschitz property of the drag force and give a proof of the uniqueness theorem. 
 
%\newpage
\section{Assumptions and Reformulations} \label{sec:Assump}
In this section we state all the assumptions used to prove the uniqueness of the solution. We also introduce several reformulations
of the density function $f$ as well as the drag term $G$. Most of the discussion here is built upon the understanding of the physics 
underlying the interactions of the gas particles with the disk. 

Throughout this paper we let $T$ be a fixed arbitrary time.  

\vspace{-2mm}

\subsection{Main Assumptions} The assumptions on the system are 
\begin{enumerate}[label=\bf{(A\arabic*)}]
\setcounter{enumi}{-1}

\item  \label{a0} 
Particles cannot penetrate the disk. \smallskip

\item  \label{a1}
\textit{Assumptions on the gas:} the initial distribution $\phi_0 = \phi_0(v)$ satisfies

\begin{enumerate}

\item $\phi_0 \in L^\8 (\R)$;

\item The zeroth, first and second moments of $\phi_0$ are finite:
\begin{align}
\int_\R (1 + v^2)\phi_0(v) \dv < \8.
\end{align}

\end{enumerate}

\item \label{a2}
\textit{Assumption on the disk:} velocity of the disk is 
locally Lipschitz with  
\begin{align} \label{const:Lip}
\| p \|_{L^\8(0, T)} + \| \dot p \|_{L^\8(0, T)} \leq M, \qquad  \end{align}
where $M$ may depend on $T$. 
\end{enumerate}

\subsection{Reformulation of the Model} \label{subsec:reform}
For the rest of the paper we will only consider the gas to the right of the disk since the analysis for the gas 
to the left of the disk is analogous. This lets us drop the sub-indices $R$ and $L$ in \eqref{BC:R}-\eqref{BC:L}
and~\eqref{def:drag-L}-\eqref{def:drag-R}.

We begin by simplifying the expression for the drag forces. The expression for the outgoing density in \eqref{BC:R} 
allows us to write
\begin{align*}
\int^\8_{p(t)} (v - p(t))^2 f^+(\Eta(t), v, t) \dv
&= 2 \vpran{\int_{-\8}^{p(t)} (p(t) - v) f^-(\eta(t), v, t) \dv} \int^\8_{p(t)} (v - p(t))^2 e^{-(v - p(t))^2} \dv \\
&= \frac{\sqrt{\pi}}{2} \int_{-\8}^{p(t)} (p(t) - v) f^-(\eta(t), v, t) \dv,
\end{align*}
so the expression for the drag force can be written as 
\begin{align} \label{def:drag}
G_p(t) &= \int^{p(t)}_{-\8} \vpran{(p(t) - v)^2 + \frac{\sqrt{\pi}}{2} (p(t) - v)} f^-(\eta(t), v, t) \dv .  
\end{align}

%\newpage
\subsubsection{Partition of the density function} \label{sec:gas_decomp}
To make the drag term more amiable to analysis, we introduce the idea of recursive scattering: for $x\neq \eta(t)$ 
let $f_n(x, v, t)$ be the density functions of the particles that have collided with the disk exactly $n$ times in the
past. Away from the disk they satisfy the same free transport equation as $f$. For $x = \Eta(t)$ we define $f_n^\pm(x, v, t)$ 
in terms of the one-sided limits similar to those in \eqref{def:f_pm}:
\begin{align} \label{def:fn_pm}
f_n^\pm(\eta(t), v, t) = \lim_{\Eps \to 0^+} f_n^{\pm}(\Eta(t)\pm \Eps v, v, t \pm \Eps).
\end{align}
The boundary conditions on $f_n$'s are similar to 
those for the full density function $f$, with the exception that the collision with the disk now increments the sub-index.
In particular, for $v > p(t)$ and $n \geq 0$ we write
\begin{align} \label{def:fn_v}
f_{n+1}^+(\eta(t), w, t) &= 2 e^{-(w - p(t))^2}  \int_{-\8}^{p(t)}\ (p(t) - v) f_{n}^-(\eta(t), v, t) \dv .
\end{align} 
We also define \frec to be the density function of the particles that have collided with the disk in the past:
\begin{gather} \label{def:frec}
\frec(x, v, t) = \sum_{n \geq 1} f_n(x, v, t) \,.
\end{gather}
Thus the full density function is decomposed as 
\begin{align*}
   f(x, v, t) = \phi_0(v) + \frec(x, v, t) \,.
\end{align*} 
Similarly, we define 
$G_{p, \text{rec}}$ to be the drag forces due to particles that have collided with the disk in the past:
\begin{align} \label{def:Grec}
G_{p, \text{rec}}(t) &= 
\int^{p(t)}_{-\8} \vpran{(p(t) - v)^2 + \frac{\sqrt{\pi}}{2} (p(t) - v)} \frec(\eta(t), v, t) \dv .
\end{align}
%and thus write the full drag term as $G_p = G_{p, 0} + G_{p, \text{rec}}$.
%\red{why do we still have $G_{p, 0}$ here?}

\subsubsection{Average Velocity}
We now address the possibility for the particles to collide with the disk multiple times. Throughout the paper we adopt
the following notation for the average velocity of the disk on the time interval $[s, t]$: \vspace{-5mm}
\begin{align*} 
   \vint{p}_{s, t} = \frac{1}{t-s} \int_s^t p(\tau) \dtau .
\end{align*}
It will play a significant role in the precollision conditions and the change of variables. We summarize a few useful
properties of the average velocity in the following lemma:

%\newpage
\begin{lem} \label{lem:technical} 
Suppose $t \in (0, T)$ and $s \in (0, t)$. Let $v(\cdot, \cdot)$ be the function defined as
\begin{gather} \label{def:preco-time}
v(s, t) = \vint{p}_{s, t}.
\end{gather}
Let $M$ be the Lipschitz constant in~\eqref{const:Lip}. Then
\begin{enumerate}[label=(\alph*), leftmargin=*]

\item \label{prop:vst_Linf}
$v(s, t)$ satisfies the bound $|v(s, t)| \leq M$; \smallskip

\item \label{prop:vst_deriv}
the derivatives of $v(s, t)$ are
\begin{align} \label{formula:deriv-v}
   \dds{v} = \frac{v - p(s)}{t-s}  ,  \qquad
   \dds[t]{v}  = \frac{p(t) - v}{t-s} ;
\end{align}

\item \label{prop:vst_deriv_bnd}
the derivatives of $v$ satisfy the following estimates:
\begin{align} \label{bnd:dv}
   \abs{\dds{v}} \leq \frac{1}{2}M , \qquad
   \abs{\dds[t]{v}} \leq \frac{1}{2}M . 
\end{align}
\end{enumerate}
\end{lem}
\begin{proof}
Part \ref{prop:vst_deriv} follows from direct computations. Parts \ref{prop:vst_Linf} and \ref{prop:vst_deriv_bnd} 
follow from Assumption \ref{a2}:
\begin{gather*}
|v(s, t)| = \abs{ \frac{1}{t-s} \int_s^t p(\tau) \dtau } \leq
\frac{1}{t-s} \int_s^t M \dtau = M \,, \\[3pt]
\abs{\frac{\del v}{\del t}} = \frac{\abs{p(t) - v(s, t)}}{t-s} 
\leq \frac{1}{(t-s)^2} \int_s^t \abs{p(t) - p(\tau)} \dtau
\leq \frac{\norm{\dot p}_{L^\infty(0, t)}}{(t-s)^2} \int_s^t (t - \tau) \dtau
= \frac{1}{2} M .
\end{gather*} 
The estimate for $\partial_s v$ is proved via a similar calculation.
\end{proof}

\subsubsection{Precollisional Velocities and Precollision Times}
In this section we prepare for the key step of change of variables. 
%More specifically, we will change the variable in the integrals from the velocity of the particle to its precollision time. As a result, when
%comparing the drag forces induced by two different velocity profiles, we will compare their integrands at the same precollision
%time rather than at the same velocity. 
To illustrate the idea of change of variables, we consider for a moment a simplified case where $\dot p(t) > 0$ for all $t$. Then for each 
$t \in (0, T)$ the average velocity $\vint{p}_{s, t}$ is strictly increasing in $s$, and thus is a bijection
between $[0, t]$ and $[\vint{p}_{0, t}, \, \vint{p}_{t, t}] = [\eta(t)/t, \, p(t)]$. This allows us to use the change of
variables $v = v(s, t) = \vint{p}_{s, t}$ in \eqref{def:fn_v} to obtain the following expression for $n \geq 1$ and $w > p(t)$:
\begin{align}
f^+_{n+1}(\eta(t), w, t) &= 2 e^{-(w - p(t))^2} \int_{\eta(t)/t}^{p(t)} (p(t) - v) f^-_n(\eta(t), v, t) \dv
\label{eq:pmon_v} \\[2pt]
&= 2 e^{-(w - p(t))^2} \int_0^t \dds{v} (p(t) - v(s, t)) f^-_n(\eta(t), v(s, t), t) \ds . \label{eq:pmon_s} 
\end{align}
One immediate advantage of expression \eqref{eq:pmon_s} is that it allows us to obtain an explicit recursive relationship
between the sequence of outgoing densities $\{f_n^+\}$. Indeed, since the distribution density does not change between
collisions, we have \vspace{-5mm}
\begin{align*}
f^-_n(\eta(t), v(s, t), t) = f^+_n(\eta(s), v(s, t), s).
\end{align*}
This in turn implies
\begin{align*}
f^+_{n+1}(\eta(t), w, t) 
= 2 e^{-(w - p(s))^2} \int_0^t \dds[\tau]{v(s, t)} (p(t) - v(s,t)) f^+_{n}(\eta(s), v(s, t), s) \ds .
\end{align*}
In Sections \ref{sec:prelim_bnds} and \ref{sec:uniq} we show the full usage of a similar recursive relation for obtaining the
estimates for the density function and the drag term.

Without the monotonicity assumption a proper change of variables requires much more work. The main difficulty 
is the non-injectivity of the mapping $v(\cdot, t)$ defined in~\eqref{def:preco-time}. To handle it, we start by identifying that, among all the particles that are to collide with the disk at time $t$, which ones have had 
%which of the particles that are to collide with the disk at time $t$ have had 
a collision in the past. Velocities of such particles will henceforth be called \textit{precollisional}, to signify that the corresponding particles 
have previously collided with the disk. They must satisfy the following condition: 
\begin{quote}
There exists time $s \in [0, t)$ such that 
the particle and the disk have travelled the same distance over $[s, t]$ and $v < p(t)$.
\end{quote}
Since the velocity of the particle does not change
between consecutive collisions, the above condition can be written as 
\begin{gather} \label{cond:same_dist} 
(t-s) v = \int_s^t p(\tau) \, \dtau \qquad \text{or} \qquad v = \vint{p}_{s, t} .
\end{gather}
Introduce the notation 
\begin{align*}
  \ukt = \min_{\mathsmaller{s \in [0, t]}} \vint{p}_{s, t}  \,.
\end{align*} 
Then the precollisional velocities can be characterized as
\begin{prop} 
Suppose a particle with velocity $v$ is colliding with the disk at time $t$ and $v \neq \ukt$. 
Then it has collided with the disk in the past if and only if
\begin{gather} \label{precol_cond}
\ukt < v < p(t).
\end{gather} 
\end{prop}
\begin{proof}
%Suppose the particle has
%collided with the disk in the past. Then since it is colliding with the disk at time $t$ and by the pre-collision condition
%\eqref{precoll_cond} we have, respectively, 
%%\vspace{-5mm}
%\begin{gather*}
%p(t) > v \qquad \text{and} \qquad v = \vint{p}_{s, t} > \vint{p}_{s, t} (t) \,.
%\end{gather*}
Let $\ukt < v < p(t)$. Since $\vint{p}_{s, t}$ is a continuous function of $s$ for any $t$, it must obtain its minimum \ukt 
at some $s^* \in [0, t]$. Assume $s^* < t$ and suppose for contradiction that the particle with velocity $v$ has not collided
with the disk in the past. Let $\omega(s)$ be the position of the particle. Then %\vspace{-5mm}
\begin{align*}
\omega(s) = \eta(t) - (t-s)v.
\end{align*}
Since the particle is colliding with the disk from the right and could not have penetrated the disk by assumption \ref{a0}, 
it must have been to the right of the disk for all $s \in [0, t)$, that is
\begin{gather*}
\omega(s) - \eta(s) \geq 0 \quad \forall s \in [0, t).
\end{gather*}
However, this condition is violated at $s = s^*$ since
\begin{align}
\omega(s^*) - \eta(s^*) &= \big[ \eta(t) - (t-s^*)v \big] - \big[ \eta(t) - (t-s^*)\ukt \big]  \nn
\\
&= \big[ \ukt - v \big](t-s^*) < 0, \label{precol_cond_strict_ineq}
\end{align}
which is a contradiction. 
If $s^* = t$, then $\ukt = \vint{p}_{t, t} = p(t)$. This again violates \eqref{precol_cond}.%in turn implies that no velocity $v$ can satisfy \eqref{precol_cond},
%and thus all particles colliding with the disk at time $t$ are colliding for the first time. 

The converse is an immediate consequence of \eqref{cond:same_dist}. 
\end{proof}
Denote the set of all possible precollisional velocities by \Vt where
\begin{gather} \label{def:Vt}
\Vt = \big( \ukt, \, p(t) \big).
\end{gather}

We now identify the times of the precollisions. 
\begin{definition}[Precollision Time] \label{def:precoll_time}
Suppose a particle with velocity $v \in \Vt$ is to collide with the disk at time~$t$. Then the time $s_v$ is a corresponding
\textit{precollision time} if $(v, s_v)$ satisfies \eqref{cond:same_dist} and the particle has been ahead of the disk for all 
$\tau \in (s, t)$, that is, %\vspace{-3mm}
\begin{gather} \label{cond:ahead}
\eta(t) - (t-\tau)v \geq \eta(\tau) \qquad \fa \tau \in (s, t).  
\end{gather}
Note that this condition is a mathematical formulation of Assumption~\ref{a0}. 
\end{definition}

Let \Nt be the set of all possible precollision times. To construct a bijection between \Vt and \Nt we need a more
explicit characterization of the latter. To this end, we first
%The first step toward obtaining it is to 
rewrite \eqref{cond:ahead} as
\begin{align} \label{cond:ahead-1}
   v \leq \vint{p}_{\tau, t}  \qquad \fa \tau \in (s, t) , \\[-25pt] \nn
\end{align}
which in turn implies \vspace{-5mm}
\begin{gather} \label{cond:ahead-2}%\label{def:vbar}
v \leq \min_{\tau \in [s, t]}\vint{p}_{\tau, t} .
\end{gather}
Motivated by~\eqref{cond:ahead-2}, we define the modified average velocity \vm:
\begin{gather} \label{def:vbar}
 \vm := \min_{\tau \in [s, t]} \vint{p}_{\tau, t}.
\end{gather} 
From the combination of \eqref{cond:ahead-2} with \eqref{cond:same_dist} we conclude that $s_v$ is a precollision time
corresponding to $v$ if and only if $v = v(s, t) = \vm$. Consequently, we have
\begin{gather} \label{def:N-t}
\Nt = \set[{s \in [0, t]}]{ v(s, t) = \vm}. 
\end{gather}
Note the function $\vm[\cdot, t]$ is monotonically, but not necessarily strictly, increasing. It can be thought of
as the tightest monotonically increasing lower envelope for $v(s, t)$; the notation \vme had been chosen to reflect that.
We give an example of \Nt and \vme in Figure~\ref{fig:vm} to help intuitive understanding of their properties.

\smallskip
\noindent \textit{Notation.\hspace{-2pt}} For a given $t$, we will use \vm[\mc A, t] to denote the image of the set $\mc A$ under the map 
\vm[\cdot, t] and $\vme^{-1}(\mc B, t)$ to denote the pre-image of the set $\mc B$ under the map \vm[\cdot, t]. 
Note that the inversion is only performed in the first variable with the second variable $t$ fixed. 

\smallskip

We now establish properties of \Nt and \vme. A large part of the analysis is essentially Riesz's rising sun lemma
\cite{leoniSS} with a sign change.
%\newpage
\begin{lem} \label{le:Ntc_open}
Fix $t \in [0, T]$ and let \Nt be the set defined in~\eqref{def:N-t}. Then $\Ntc = [0, t] \setminus \Nt$ is open in $[0, t]$.   
\end{lem}
\begin{proof} 
Note that since $\vm[t, t] = v(t, t) = p(t)$, we have $t \in \Nt$. We write \Ntc as
\begin{gather*}
\Ntc = \set[{s \in [0, t)}]{ v(s, t) > \vm }
     = \set[{s \in [0, t)}]{ v(s, t) > v(\tau, t) \text{ for some } \tau \in (s, t) }
     = \bigcup_{\tau \in [0, t)} \hspace{-2mm} \Or_\tau, \\[-25pt]
\end{gather*}
where \dm{ \Or_\tau = [0, \tau) \cap v^{-1}(\left( v(\tau, t), \8 \right), t)}. Since $v(\cdot, t)$ is continuous, 
the pre-image $\Or_\tau$ is open in the subspace topology on $[0, t]$. Therefore \Ntc is an open subset of $[0, t]$. 
\end{proof}

\vspace{-4mm}

\begin{figure}[b] 
\centering
\captionsetup{width=0.75\linewidth} 
\includegraphics[scale=0.25]{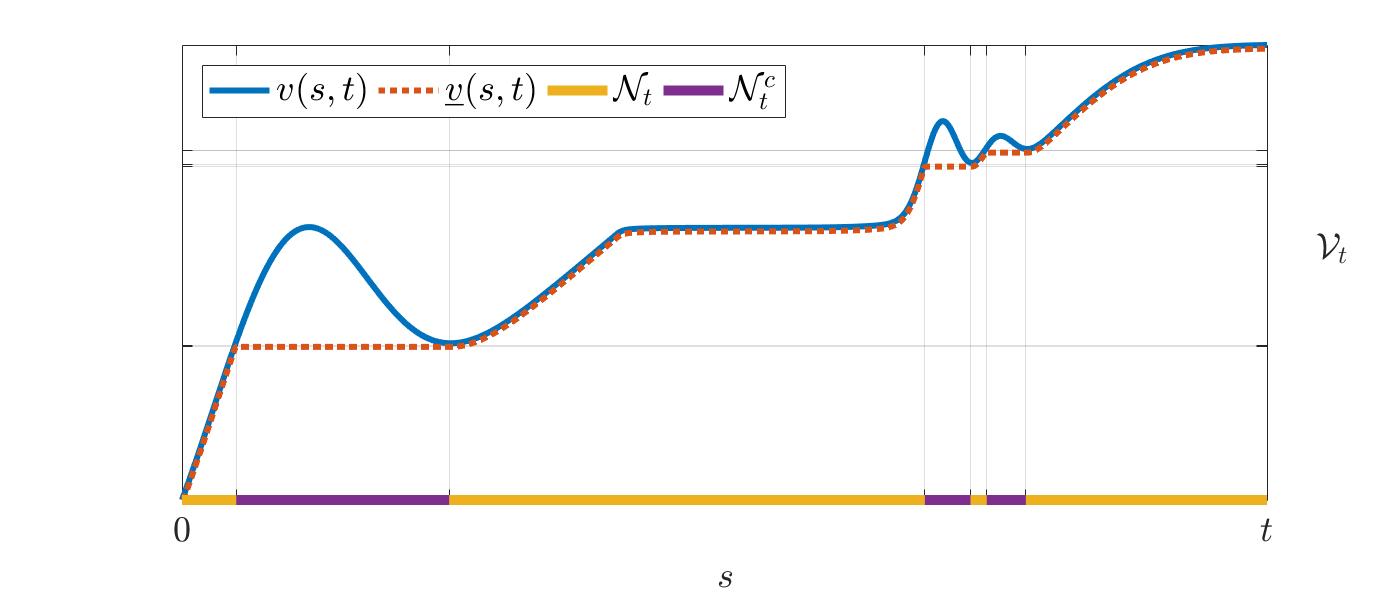} 
\vspace{-2mm} 
\caption{An example of the average velocity $v(s, t)$ and the corresponding modified average velocity \vm, with $t$ fixed.
Note how the difference between the two informs the separation of the time domain into \Nt and \Ntc.} 
\label{fig:vm}
\end{figure}

%\newpage
\begin{lem}[Properties of \vme] \label{le:vmprops} 
Let $\vm$ be the modified velocity defined in~\eqref{def:vbar}.  Then 
\begin{enumerate}[label=(\alph*), leftmargin=*]

\item \label{le:cont} 
$\vm[\cdot, t]: [0, t] \to \Vt$ is continuous. \Ss

\item \label{le:conn_comp}
Let $(a, b)$ be a maximal connected component of \Ntc. Then for all $s \in (a, b)$ we have 
\begin{gather*}
\vm[a, t] = \vm[s, t] = \vm[b, t].
\end{gather*}
Consequently, $\del_s \vm = 0$ on $\Nt^c$.  \Ss

\item \label{le:range}
Restricting \vme to \Nt does not change its range: $\vm[{[0, t]}, t] = \vm[\Nt, t]$.  \Ss

\item \label{le:Lip_s} 
For any fixed $t \in [0,T]$, \vm is Lipschitz in $s$ with $\abs{\del_s \vm[s,t]} \leq M/2$ for almost every $s \in [0, T]$. \Ss 

\item \label{le:Lip_t} 
For any fixed $s \in [0, t]$, \vm is Lipschitz in $t$ with $\abs{\del_t\vm[s, t]} \leq M$ for almost every $t \in [0, T]$. \Ss

\item \label{le:Dt}
Let $\mc{L}(A)$ be the Lebesgue measure of $A$ and define
\begin{gather*}
\mc D_t := \set[{s \in [0, t]}]{ \dds{\vm} \text{ exists} }. \\[-28pt]
\end{gather*}
Then $ \mc{L} ( \vm[ \mc D_t^c, t ] ) =  \mc{L} ( v(\mc D_t^c, t ) ) = 0 $. \Ss 

\item \label{le:deriv-vm} 
For all $s \in \Nt \cap \mc D_t$ we have $ \del_s \vm = \del_s v(s, t)$. %\Ss

\end{enumerate}
\end{lem}

%\begin{figure}[t] 
%\centering
%\captionsetup{width=0.7\linewidth}
%%\includegraphics[width=4.5in]{../Global/vbar.jpg}
%\includegraphics[scale=0.2]{../Global/vbar2.jpg}  
%\caption{An example of the average velocity $v(s, t)$ and the corresponding modified average velocity \vm, with $t$ fixed.
%Note how the difference between the two informs the separation of the time domain into \Nt and \Ntc.} 
%\label{fig:vm}
%\end{figure}

\begin{proof} 
\ref{le:cont}
Let $\Eps > 0$ be given. By the uniform continuity of $v(\cdot, t)$ on $[0, t]$, there exists $\delta > 0$ such that 
\begin{align} \label{bound:unif-v-s-t}
   |v(s, t) - v(s', t)| < \Eps  \quad \text{whenever} \quad |s-s'|<\delta .
\end{align} 
Without loss of generality, suppose $0 \leq s - s' \leq \delta$. Then by definition of \vm[\cdot, t] we have
\begin{align*}
  0 
\leq \vm - \vm[s', t]
&= \vm - \min \{\min_{\tau \in [s', s]}v(\tau, t), \,\, \vm\} 
\\
&= \vpran{\vm[s, t] - \min_{\tau \in [s', s]} v(\tau, t)} 
     \Sgn\Big(\vm[s, t] - \min_{\tau \in [s', s]} v(\tau, t) \Big)
\\
&\leq \Big|v(s, t) - \min_{\tau \in [s', s]} v(\tau, t)\Big| < \Eps ,
\end{align*}
where the last inequality follows from~\eqref{bound:unif-v-s-t}.

\medskip
\noindent
\ref{le:conn_comp} 
Since $\vm[\cdot, t]$ is non-decreasing we have $\vm \leq \vm[b, t]$. Suppose for contradiction that $\vm < \vm[b, t]$. 
Then there must exist $\tau \in [s, b)$ such that $\vm = v(\tau, t)$, which in turn implies that $\vm[\tau, t] = v(\tau, t)$.
But then $\tau \in \Nt$ by the definition of $\Nt$, which is a contradiction since $\tau \in (a, b) \subseteq \Ntc$. 
Equality $\vm[a, t] = \vm$ follows from continuity of \vm[\cdot, t]. 
\medskip

\noindent 
\ref{le:range} Take $s \in \Ntc$ and let $(a, b)$ be the largest connected subset of $\Ntc$ containing it. 
By part \ref{le:conn_comp}, we have $\vm = \vm[b, t]$. Thus $\vm[s, t] \in \vm[\Nt, t]$ since $b \in \Nt$. 

\medskip
%\newpage
\noindent 
\ref{le:Lip_s} For $s \in \Ntc$ let $\left(a_s,  b_s \right)$ be the largest connected component of \Ntc 
containing $s$. In other words, the lower bound $a_s$ is the largest time less than $s$ such that $\vm[a_s, t] = v(a_s, t)$. Similarly, the upper bound $b_s$ is the smallest time greater than $s$ such that $\vm[b_s, t] = v(b_s, t)$. For $s \in \Nt$ we simply let $a_s = b_s = s$. Then for both cases we have \vspace{-5mm}
\begin{gather*}
v(a_s, t ) = \vm[a_s, t] = \vm[s, t] = \vm[b_s, t] = v(b_s, t ). 
\end{gather*}

\noindent
Let $\tau, \tau' \in [0, t]$ and assume, without loss of generality, that $\tau \leq \tau'$. If $(\tau, \tau') \ins \Ntc$ 
then $v(\tau', t) - v(\tau, t) = 0$ by part \ref{le:conn_comp}. Otherwise, by Lemma~\ref{lem:technical} we have 
\begin{gather*}
|\vm[\tau', t] - \vm[\tau, t] | = \abs{v(a_{\tau'}, t) - v(b_\tau, t) } 
\leq \frac{M}{2} \left(a_{\tau'} - b_\tau \right) 
\leq \frac{M}{2} \left(\tau' - \tau \right). 
\end{gather*}

\noindent
\ref{le:Lip_t} Let $t'> t$ and fix $s \in [0, t]$. Since $v(s, t)$ is continuous, $\vm = v(\tau, t)$ for some $\tau \in [s, t]$. 
We have
\begin{gather*}
\vm[s, t'] \leq v(\tau, t') \leq v(\tau, t) + \frac{M}{2}|t'-t| \ \implies \ \vm[s, t'] - \vm \leq \frac{M}{2}|t'-t|.
\end{gather*}
On the other hand, for all $\tau \in [0, t]$ we have
\begin{gather*}
    v(\tau, t') 
\geq v(\tau, t) - M|t'-t| 
\ \implies \ 
\vm[s, t'] \geq \vm - M|t'-t|,
\end{gather*}
Thus the function \vm is Lipschitz it $t$. Consequently, $\partial_t \vm$ exist for almost all $t$ and 
$\abs{\del_t\vm[s, t]} \leq M$.

\medskip
\noindent
\ref{le:Dt}
Since \vm is Lipschitz in $s$, it is almost everywhere differentiable and thus $\mc{L}(\mc D^c_t) = 0$.  
Since \vm[\cdot, t] is absolutely continuous, it possesses the Lusin property:
$ \mc{L} ( \vm[ \mc D_t^c, t ] ) = 0 $. The same argument holds for $v(\cdot, t)$. \vspace{-4mm}

\noindent
\ref{le:deriv-vm} 
Let $s \in \Nt \cap \mc D_t$. Then \dds{\vm} is given by the definition of the classical derivative. Therefore, 
\begin{align*}
\dds{\vm} &= \lim_{\tau \to s^+} \frac{\vm[\tau, t] - \vm[s, t]}{\tau - s}
           = \lim_{\tau \to s^+} \frac{\vm[\tau, t] - v(s, t)}{\tau - s}
              \leq \lim_{\tau \to s^+} \frac{v(\tau, t) - v(s, t)}{\tau - s}
                 = \dds{v(s, t)}, \\[5pt]
\dds{v(s, t)} &= \lim_{\tau \to s^-} \frac{v(s, t) - v(\tau, t)}{s - \tau}
               = \lim_{\tau \to s^-} \frac{\vm     - v(\tau, t)}{s - \tau}  
            \leq \lim_{\tau \to s^-} \frac{\vm - \vm[\tau, t]}{s - \tau}  
               = \dds{\vm} .
\end{align*}
It now follows that $ \del_s \vm = \del_s v(s, t)$.
\end{proof}

Since the measure of the set $\mc D^c_t$, as well as its images under both \vm[\cdot, t] and $v(\cdot, t)$, is zero, we can 
safely ignore it from now on. 

From Lemma~\ref{le:vmprops}\ref{le:range} it follows that the map $\vm[\cdot, t] : \Nt \to \Vt$ is a surjection. 
However, it is not necessarily an injection, so further restriction is required. To show that the restriction we
are about to make does not affect the dynamics of the disk we will need the following lemma from \cite{leoniSS} (page 77): 
\begin{lem}[\cite{leoniSS}] \label{lem:measure}
Let $I \ins \R$ be an interval and let $u: I \to \R$. Assume that there exists a set $E \ins I$ (not necessarily measurable)
and $M \geq 0$ such that $u$ is differentiable for all $x \in E$, with
\begin{gather*}
| u'(x) | \leq M \qquad \text{for all } x \in E.
\end{gather*}
Then $\mc{L}_\circ (u(E)) \leq M\mc{L}_\circ (E)$, where $\mc{L}_\circ$ denotes the outer Lebesgue measure. 
\end{lem}
We are now ready to make the restriction and create a bijection.
\begin{thm} \label{thm:vbar}
For any fixed $t \in [0, T]$, let \vm[\cdot, t] be the function defined in~\eqref{def:vbar}. Let
\begin{gather} \label{def:Phit_Wt}
\Phi_t := \set[{ s \in [0, t] }]{ \dds{\vm} > 0 } \qquad \text{and} \qquad \mc{W}_t := \vm[\Phi_t , t]. 
\end{gather}
Then $\vm[s, t] = v(s, t)$ for all $s \in \Phi_t$, the mapping $\vm[\cdot, t]: \Phi_t \to \mc W_t $ is a bijection and is
strictly increasing, and $\mc W_t$ contains almost all postcollisional velocities, that is, 
$\mc{L}(\Vt \! \setminus \! \mc{W}_t) = 0 $.
\end{thm}
\begin{proof}
From Lemma~\ref{le:vmprops}\ref{le:conn_comp} we know that $\del_s \vm \equiv 0$ for all $s \in \Ntc$, so it must be the case
that $\Phi_t \ins \Nt$. Furthermore, since \vm[\cdot, t] is a monotonically increasing function on the interval $[0, t]$, its
restriction to $\Phi_t$ is strictly increasing and thus is a bijection between its domain and range. \noindent
Choosing $I = [0, t]$, $u = \vm[\cdot, t]$, $E = \Phi_t^c$ and $M=0$ in Lemma~\ref{lem:measure} gives
\begin{gather*}
\mc{L}_\circ (\mc{W}_t^c ) = \mc{L}_\circ ( \vm[\Phi_t^c, t] ) = 0.
\end{gather*}
Hence $\mc{L}_\circ(\Vt \! \setminus \! \mc{W}_t) = 0 $, so $\Vt \! \setminus \! \mc{W}_t$ is measurable and almost all postcollisional velocities are included in $\mc{W}_t$.
\end{proof}

\begin{rmk}
We have not yet discussed the relationship between the velocity of the particle that had precollided with the disk at
time $s$ and the velocity of the disk itself at time $s$; one would expect the particle to be moving faster in that case. 
Indeed, combining \eqref{cond:ahead} with \eqref{cond:same_dist} yields  
\begin{gather*}
\eta(s) + (\tau-s)v \geq \eta(\tau) \qquad \fa \tau \in (s, t),
\end{gather*}
which can in turn be rewritten as $ v \geq \vint{p}_{s, \tau}$ for all $\tau \in (s, t)$. Letting $\tau \to s$ gives
$v \geq \vint{p}_{s, s} = p(s)$, so the particle is, at least, can not be slower than the disk. The case $v = p(s)$ is the 
grazing precollision. Since 
\begin{gather*}
\dds{v(s, t)} = \frac{v(s, t) - p(s)}{t-s},
\end{gather*}
all velocities that have had a grazing precollision and their corresponding (non-unique!) precollision times are collected in 
$\mc W_t^c$ and $\Phi_t^c$ respectively. Since $ \mc{L} ( \vm[\Phi_t^c, t] ) = 0 $, particles that have had a grazing precollision 
have no effect on the dynamics of the disk, and thus can be safely excluded.  
\end{rmk}

\subsubsection{Change of Variables}
We now make a change variables in \eqref{def:Grec}: by~\eqref{def:Vt} and Theorem~\ref{thm:vbar}, we have
\begin{align}
G_{p, \text{rec}}(t) 
&= \int_{\Vt} \vpran{(p(t) - v)^2 + \frac{\sqrt{\pi}}{2} (p(t) - v)} \frec^-(\eta(t), v, t) \dv \nn 
\\[3pt]
&= \int_{\mc W_t} \vpran{(p(t) - v)^2 + \frac{\sqrt{\pi}}{2} (p(t) - v)} \frec^-(\eta(t), v, t) \dv \nn 
\\[3pt]
&= \int_{\Phi_t} \dds{\vm} \vpran{(p(t) - v(s, t))^2 + \frac{\sqrt{\pi}}{2} (p(t) - v(s, t))} 
\frec^-(\eta(t), v(s, t), t) \ds.  \label{formula:change-variable-1} %\nn
\end{align}
Furthermore, since \dds{\vm} vanishes on $\Phi_t^c$, we can write
\begin{align}
G_{p, \text{rec}}(t)  
&=  \int_{\Phi_t} \dds{\vm} \vpran{(p(t) - v(s, t))^2 + \frac{\sqrt{\pi}}{2} (p(t) - v(s, t))} 
\frec^-(\eta(t), v(s, t), t) \ds \nn 
\\[3pt]
&= \int_{\Phi_t} \ (\ldots) \ \ds + \int_{\Phi_t^c} \ (\ldots) \ \ds \nn
%\label{eq:plus0} 
\\[3pt]
&= \int_0^t \dds{\vm} \vpran{(p(t) - v(s, t))^2 + \frac{\sqrt{\pi}}{2} (p(t) - v(s, t))} \frec(s, t) \ds ,
\label{eq:vm_once}
\end{align} 
%
%\Ss\noindent
where, with a slight abuse of notation, we have written
\begin{align} \label{def:f-n-abbre}
\frec(s, t) = \sum_{n\geq1} f_n(s, t) = \sum_{n\geq1} f^-_n(\eta(t), v(s, t), t) = \sum_{n\geq1} f^+_n(\eta(s), v(s, t), s) \,.
\end{align}
The last equality in~\eqref{def:f-n-abbre} holds because the distribution density does not change between collisions. 
Note that in \eqref{eq:vm_once}, the modified velocity \vm needs to appear only in the derivative since whenever 
$\partial_s \vm \neq 0$, we have $\vm = v(s, t)$. 

Making the same change of variables in \eqref{def:fn_v} for $n\geq 1$ and using the notation in~\eqref{def:f-n-abbre} lead us 
to the key recurrence relation: 
\begin{align} \label{eq:rec}
f_{n+1}(s, t) &= 2 e^{-(v(s, t) - p(s))^2} \int_0^s \dds{\vm[\tau, s]} (p(s) - v(\tau, s)) f_n(\tau, s) \dtau. 
\end{align}

For future convenience, we define the density flux of $f_n$ as
\begin{gather} \label{def:j-n}
j_n(s) = \int_0^s \dds{\vm[\tau, s]} (p(s) - v(\tau, s)) f_n(\tau, s) \dtau,
\end{gather}
which allows us to write \vspace{-2mm}
\begin{gather} \label{eq:rec-1}
f_{n+1}(s, t) = 2 e^{-(v(s, t) - p(s))^2} j_n(s).
\end{gather}

The change of variables does not apply to the particles that have not collided with the disk previously; since such particles
maintain the initial density distribution, we have
\begin{align} 
f_1(s, t) &= 2 e^{-(v(s, t) - p(s))^2} \int^{\vm[0, s]}_{-\8} (p(s) - v) \phi_0(v) \dv, \label{def:f-1}
\\[2pt]
G_{p, 0}(t) &= \int^{\vm[0, t]}_{-\8} \vpran{(p(t) - v)^2 + \frac{\sqrt{\pi}}{2} (p(t) - v)} \phi_0(v) \dv .
\label{def:G-p-0}
\end{align}
Recalling the definition of $G_{p, \text{rec}}$ in~\eqref{def:Grec}, we have constructed a decomposition of the drag force:
\begin{align*}
   G_p(t) = G_{p, 0}(t) + G_{p, \text{rec}} .
\end{align*}
\begin{rmk}
Note that even the particles that had no precollisions obey Assumption \ref{a0} or, equivalently, 
the mathematical formulation in~\eqref{cond:ahead}. This is why the effective integration domain in $G_{p, 0}$ and 
$f_1(s, t)$ is $(-\infty, \vm[0, t])$ instead of $(-\infty, p(t))$: the latter would allow the particles originally in front
of the disk to fall behind the disk.
\end{rmk}

\section{Preliminary Bounds} \label{sec:prelim_bnds}
For future convenience we define 
\begin{gather} \label{def:alpha_n}
\alpha_n(s) = \frac{\left( 2M^2 s\right)^n}{n!}, \qquad \alpha_0(s) \equiv 1, \qquad \alpha_{-1}(s) \equiv 0.
\end{gather}

\noindent
In this section we use the recurrence relation \eqref{eq:rec} to derive essential bounds on $f_n$ and its derivatives;
they are summarized in two propositions. 
\begin{prop} \label{bound:f-n-1}
Let $j_n$ and $f_n$ be the iterative sequences given by~\eqref{def:j-n} and~\eqref{eq:rec}-\eqref{def:f-1} respectively. 
Let $M$ be the Lipschitz bound in Assumption~\ref{a2}. Then there exists a constant $Q_1$ that does not depend on $n$ such 
that for any $n \geq 1$ we have 
\begin{align} \label{bound:f-n-L-infty}
0 \leq f_n(s, t) \leq Q_1 \alpha_{n-1}(s)
\qquad \text{and} \qquad 
0 \leq j_n(s) \leq \frac{1}{2}Q_1 \alpha_n(s) .
\end{align}
Moreover, for each $s \in [0, t]$, the function $f_n(s, \cdot) \in C^1([0, T])$ with the bound
\begin{gather} \label{bound:f-n-C-1}
\abs{\dds[t]{} f_n(s, t)} \leq 4M^2 Q_1 \alpha_{n-1}(s). 
\end{gather}
As a consequence, the function $\frec$ defined in~\eqref{def:frec} satisfies
\begin{align} \label{bound:f-n-C-1}
0 \leq \frec(s, t) \leq Q_1 e^{2M^2 s} 
\qquad \text{and} \qquad 
\abs{\dds[t]{} \frec(s, t)} \leq 4M^2 Q_1 e^{2M^2 s}. 
\end{align}
\end{prop}

\begin{proof} 
First we derive the bound~\eqref{bound:f-n-L-infty}. For $n=1$ we use the definition of $f_1$ in~\eqref{def:f-1} to write
\begin{gather*}
0 
\leq f_1(s, t) 
\leq 2\int_{-\8}^{\vm[0, s]} \left( p(s) - v \right) \phi_0(v) \dv 
\leq 2\int_{-\8}^M \left( M - v \right) \phi_0(v) \dv 
=: Q_1 .
\end{gather*}
For $n \geq 2$ we apply~\eqref{eq:rec} together with the definition of $\alpha_n$ in~\eqref{def:alpha_n}:
\begin{align*}
f_n(s, t) &= 2e^{ -( v(s, t) - p(s) )^2 } \int_0^{s} \dds[\tau]{\vm[\tau, s]} (p(s) - v(\tau, s)) f_{n-1} (\tau, s) \dtau \\
&\leq 2M^2 \int_0^{s} Q_1 \alpha_{n-2} (\tau) \dtau = Q_1\alpha_{n-1}(s). 
\end{align*}
Note that the above step also gives the bound of $j_n$. Indeed, by its definition,
\begin{align*}
j_n(s) \leq \int_0^{s} \abs{\dds[\tau]{\vm[\tau, s]}} \abs{p(s) - v(\tau, s)} f_{n} (\tau, s) \dtau
\leq
M^2 \int_0^s Q_1 \alpha_{n-1}(\tau) \dtau
= \frac{1}{2}Q_1 \alpha_n(s) \,.
\end{align*}
The bound~\eqref{bound:f-n-C-1} follows directly from the definition of $f_n$. Indeed, 
\begin{gather*}
\abs{\dds[t]{}f_n(s, t)} 
= 2 \abs{\dds[t]{} e^{ -( v(s, t) - p(s) )^2 } j_{n-1}(s)}
= 2 \abs{v(s, t) - p(s)} \abs{\dds[t]{v(s, t)}} f_n(s, t) \leq 4M^2 Q_1 \alpha_{n-1}(s). 
\end{gather*}
The estimates for $\frec$ and $\del_t \frec$ follow by summing the bounds for $f_n$ and $\del_t f_n$.
\end{proof}

%The derivation of the Lipschitz bound for $f_n(s, t)$ with respect to the first variable is more involved. The result states
\begin{prop} \label{prop:bound-f-n-s}
For all $t \in [0, T]$ the function $f_n(s, t)$ is Lipschitz in $s$. As a result, it is almost everywhere differentiable in 
$s \in [0, t]$. Moreover, there exists a positive constant $Q_3$ that does not depend on $n$ such that 
\begin{align}
\left| \dds{f_{n+1}(s, t)} \right| & \leq 3^n Q_3 \vpran{\alpha_n(s) + \alpha_{n-1}(s)} ,
\qquad n \geq 0, \label{bound:f-n-deriv-s-simp}
\end{align}
As a consequence, the function $\frec$ defined in~\eqref{def:frec} satisfies
\begin{align} \label{bnd:frec_s}
\left| \dds{\frec(s, t)} \right| & \leq 4Q_3e^{6M^2 s}.
\end{align}
\end{prop}

\begin{proof}
Fix $t \in [0, T]$. %We use an induction proof to show~\eqref{bound:f-n-deriv-s-simp}, starting with the cases where $n = 0, 1$.
For $n=0$, the definition of $f_1(s, t)$ in~\eqref{def:f-1} shows it is Lipschitz in $s$ since $v(s, t)$, $p(s)$, and 
\vm[0, s] are all Lipschitz in $s$. 
This allows us to obtain the desired bound by a direct calculation:
\begin{align*}
\abs{\dds{}f_1(s, t)} 
&= \abs{\dds{}\left( 2e^{ -(v(s, t) - p(s))^2 } \int_{-\8}^{\vm[0, s]} (p(s) - v) \phi_0(v) \dv \right)} 
\\[5pt]
&\leq \abs{2(v(s, t) - p(s)) \left( \dds{v(s, t)} - \dot p(s) \right)} f_1(s, t) 
\\[5pt]
& \quad \, + 2e^{ -(v(s, t) - p(s))^2} \abs{\dds{\vme(0, s)} \left( p(s) - \vm[0, s] \right)} \phi_0(\vme(0, s)) 
\\[5pt]
& \quad \, + 2e^{ -(v(s, t) - p(s))^2 } \int_{-\8}^{\vm[0, s]} \abs{\dot p(s)} \phi_0(v) \dv 
\\[5pt]
&\leq 8M^2 Q_1 + 4M^2 \| \phi_0 \|_\8 + 2M \| \phi_0 \|_1 =: Q_2.
\end{align*}
We now proceed by induction. Assume that the conclusion holds for $f_n$. Without loss of generality, assume 
$0 \leq s < s' \leq t$. Then
%\vspace{-3mm}
\begin{align*}
j_n(s) - j_n(s') &= \int_0^{s} \dds[\tau]{\vme(\tau, s)}(p(s) - v(\tau, s)) f_n (\tau, s) \dtau
- \int_0^{s'} \dds[\tau]{\vme(\tau, s')}(p(s') - v(\tau, s')) f_n (\tau, s') \dtau 
\\[2pt]
&= \int_{s'}^{s} \dds[\tau]{\vme(\tau, s)}(p(s) - v(\tau, s)) f_n (\tau, s) \dtau \\[2pt]
& \quad \, + \int_0^{s'} \left[ \dds[\tau]{\vme(\tau, s)} - \dds[\tau]{\vme(\tau, s')} \right]
(p(s) - v(\tau, s)) f_n (\tau, s) \dtau \\[2pt]
& \quad \, + \int_0^{s'} \dds[\tau]{\vme(\tau, s')} \big[(p(s) - v(\tau, s)) - (p(s') - v(\tau, s')) \big] f_n (\tau, s) \dtau \\[2pt]
& \quad \, + \int_0^{s'} \dds[\tau]{\vme(\tau, s')} (p(s') - v(\tau, s')) \big[ f_n (\tau, s) - f_n (\tau, s') \big]\dtau %\\[5pt]
=: I_1 + I_2 + I_3 + I_4.
\end{align*}

\noindent
By Lemma~\ref{le:vmprops}(d) and Proposition~\ref{bound:f-n-1}, we obtain estimates of $I_1$, $I_3$ and $I_4$ as follows:
\begin{align*}
\abs{I_1} 
&= \abs{\int_{s'}^{s} \dds[\tau]{\vme(\tau, s)}(p(s)-v(\tau, s)) f_n (\tau, s) \dtau} 
\leq |s-s'| M^2 Q_1 \alpha_{n-1}(s), \\%[5pt]
\abs{I_3} 
&= \abs{\int_0^{s'} \dds[\tau]{\vme(\tau, s')} \left[(p(s) - v(\tau, s)) - (p(s') - v(\tau, s')) \right] 
f_n (\tau, s) \dtau} \\%[5pt]
&\leq |s-s'|M^2\int_0^{s'}Q_1 \alpha_{n-1}(\tau) \dtau \leq |s-s'|\frac{Q_1}{2} \alpha_n(s'), %\\[5pt]
\end{align*}
and
\begin{align*}
%\intertext{and}
%
\hspace{14pt}
\abs{I_4} 
&= \abs{\int_0^{s'} \dds[\tau]{\vme(\tau, s')} (p(s') - v(\tau, s')) \left[ f_n (\tau, s) - f_n (\tau, s') \right]\dtau} 
\\%[5pt]
&\leq M^2\int_0^{s'} \abs{f_n (\tau, s) - f_n (\tau, s')}\dtau 
\leq M^2 |s-s'| \int_0^{s'} 4M^2Q_1 \alpha_{n-1}(\tau)  \dtau 
\\[2pt]
&\leq |s-s'| 2M^2Q_1 \alpha_n(s').
\end{align*}

To bound $I_2$ we note that by Lemma~\ref{le:vmprops}(d) and the induction hypothesis on $f_n$, the integrands $\vm[\cdot, s]$,
 $\vm[\cdot, s']$ and $(p(s) - v(\cdot, s)) f_n (\cdot, s)$ are all Lipschitz. Hence we can integrate by parts and obtain
\begin{align*}
I_2 &= \int_0^{s'} \left[ \dds[\tau]{\vme(\tau, s)} - \dds[\tau]{\vme(\tau, s')} \right]
(p(s) - v(\tau, s)) f_n (\tau, s) \dtau \\[5pt] 
&= \big[ \left( \vme(\tau, s) - \vme(\tau, s') \right) (p(s) - v(\tau, s)) f_n (\tau, s)\big]_{\tau=0}^{\tau=s} \\[5pt]
& \quad \, + \int_0^{s'} \big[ \vme(\tau, s) - \vme(\tau, s') \big] \dds[\tau]{v(\tau, s)} f_n (\tau, s) \dtau \\[5pt] 
& \quad \, - \int_0^{s'} \big[ \vme(\tau, s) - \vme(\tau, s') \big] (p(s) - v(\tau, s)) \dds[\tau]{f_n (\tau, s)} \dtau \,.
\end{align*} %\\[5pt]
This gives the bound
\begin{align*}
\abs{I_2}
&\leq 2 M | \vme(0, s) - \vme(0, s') |Q_1 \delta_{1n} + \frac{M^2}{2}|s-s'|\int_0^{s'} Q_1 \alpha_{n-1}(\tau) \dtau 
+ 2 |s-s'|M^2 \int_0^{s'} \left| \dds[\tau]{f_n (\tau, s)} \right| \dtau \\[5pt]
&\leq 2 |s-s'|M^2 Q_1 \delta_{1n} + \frac{Q_1}{4}|s-s'| \alpha_n(s') 
+ 2 |s-s'|M^2 \int_0^{s'} \left| \dds[\tau]{f_n (\tau, s)} \right| \dtau,
\end{align*}
where $\delta_{1n}$ is the Kronecker delta: $\delta_{1n} = 1$ when $n=1$ and vanishes otherwise. 
Combining the estimates for $I_1$-$I_4$, we have
\begin{gather*}
  \frac{| j_n(s) - j_n(s') |}{|s-s'|} 
\leq 
  Q_1 \alpha_n(s) \left(  \frac{3}{4} 
  + 2 M^2 \right) 
  + 2 M^2 Q_1 \alpha_{n-1}(s') + M^2 Q_1 \delta_{1n} 
  + 2 M^2 \int_0^{s'} \left| \dds[\tau]{f_n (\tau, s)} \right| \dtau \,.
\end{gather*}
The right-hand side of the inequality above is bounded uniformly in $s$ and $s'$ since 
$\partial_\tau f_n (\tau, s) \in L^\infty(0, s)$ by the induction assumption.
Therefore, $j_n(s)$ is Lipschitz, and thus differentiable almost everywhere with
\begin{align*}
   \left| \dds{j_n(s)} \right| 
&= \lim_{s' \to s} \frac{| j_n(s) - j_n(s') |}{|s-s'|} 
\\
&\leq 
   Q_1 \alpha_n(s) \left(  \frac{3}{4} 
   + 2M^2 \right) 
   + M^2 Q_1 \alpha_{n-1}(s) 
   + 2 M^2 Q_1 \delta_{1n} 
   + 2 M^2 \int_0^{s} \left| \dds[\tau]{f_n (\tau, s)} \right| \dtau.
\end{align*}
To derive the Lipschitz bound for $f_{n+1}$ we separate the two cases where $n=1$ and $n \geq 2$. For $n=1$ we have
\begin{align*}
  \abs{\dds{j_1(s)}} 
&\leq 
  Q_1 \alpha_1(s) \left( \frac{3}{4} + 2M^2 \right) 
  + M^2 Q_1 
  + 2 M^2 Q_1 
  + 2 M^2 \int_0^{s} Q_2 \dtau \\[5pt]
&\leq  
   \alpha_{1}(s) 
      \left( \frac{3}{4} Q_1 + 2M^2 Q_1 + Q_2 \right) 
   + 3M^2 Q_1  .
%= \alpha_{1}(s) Q_3 + 3M^2 Q_1.
\end{align*}
Applying the bound above in the definition of $f_2$ gives
\begin{align*}
\abs{\dds{f_2(s, t)}} 
&= \abs{\dds{}\left( 2e^{ -(v(s, t) - p(s))^2 } j_1(s) \right)}
\\[5pt]
&\leq  4 \abs{v(s, t) - p(s)} \abs{ \dds{v(s, t)} - \dot p(s)} 
e^{ -(v(s, t) - p(s))^2 } \abs{j_1(s)} + 
2e^{ -(v(s, t) - p(s))^2} \abs{\dds{j_1(s)}} \\[5pt]
&\leq 
  6 M^2 Q_1 
  + \alpha_1(s)
     \vpran{ 8M^2Q_1 
                 + 2 \vpran{\frac{3}{4} Q_1 + 2M^2 Q_1 + Q_2}} .
\end{align*}
For $n \geq 2$, by using the bounds for $f_{n+1}$ and $\partial_s j_n$ we have
\begin{align*}
    \abs{\dds{f_{n+1}(s, t)}}
&= \abs{2 \left( \dds{} e^{-(v(s, t) - p(s))^2 } \right) j_n(s) 
     + 2 e^{ -(v(s, t) - p(s))^2 } \dds{} j_n(s)} \\[5pt]
&\leq 
     2 \abs{v(s,t) - p(s)} \abs{ \dds{v(s,t)} - \dot p(s)} f_{n+1}(s,t) 
     + 2 e^{ -(v(s,t) - p(s))^2 } \abs{\dds{} j_n(s)} \\[5pt]
&\leq 
   \alpha_n(s) Q_1\left( 16M^2 + \frac{3}{2} \right) 
   + 2 M^2 Q_1 \alpha_{n-1}(s) 
   + 4 M^2 \int_0^{s} \left| \dds[\tau]{f_n (\tau, s)} \right| \dtau . 
\end{align*}
Using the induction assumption on $f_n$, the last integral term is bounded as
\begin{align*}
   4 M^2 \int_0^{s} \left| \dds[\tau]{f_n (\tau, s)} \right| \dtau
&\leq 
   4M^2 Q_3 3^{n-1} \int_0^{s}  \vpran{\alpha_{n-1}(\tau) + \alpha_{n-2}(\tau)} \dtau 
\\ %[5pt]
&= 2 \cdot 3^{n-1} Q_3 \vpran{\alpha_{n-1}(s) + \alpha_n(s)} .
\end{align*}
Hence, if we choose 
\begin{align*}
   Q_3 
= Q_1\vpran{16M^2 + \frac{3}{2} + 2 M^2 Q_1} + 2 Q_2 ,
\end{align*}
then for $n \geq 2$, we have
\begin{align*}
     \abs{\dds{f_{n+1}(s, t)}}
\leq 
    Q_3  \vpran{\alpha_{n-1}(s) + \alpha_n(s)}
    + 2 \cdot 3^{n-1} Q_3 \vpran{\alpha_{n-1}(s) + \alpha_n(s)}
%\\
= 3^n Q_3 \vpran{\alpha_{n-1}(s) + \alpha_n(s)} ,
\end{align*}
which finishes the induction proof. Since $Q_3 > Q_2$, the bound above holds for $n=0$ as well. 

The Lipschitz estimate~\eqref{bnd:frec_s} follows by summing the bounds \eqref{bound:f-n-deriv-s-simp}:
\begin{gather*}
\left| \dds{\frec(s, t)} \right| \leq \sum_{n=0}^\8 \left| \dds{f_{n+1}(s, t)} \right| 
\leq \sum_{n=0}^\8 3^n Q_3 \vpran{\alpha_{n-1}(s) + \alpha_n(s)} = 4Q_3e^{6M^2 s}. \qedhere
\end{gather*}
\end{proof}

\vspace{1mm}
%\newpage
\section{Proof of Uniqueness} \label{sec:uniq}
In this section we prove the uniqueness theorem. An essential preliminary result is a Lipschitz bound 
for the density functions corresponding to different disk dynamics. Recall that $\|\cdot\|$ denotes the $L^\infty$-norm unless
otherwise specified. We begin by showing that modified average velocity satisfies a Lipschitz bound

\begin{lem} \label{le:Lip_p} 
Let $p$ and $q$ be two Lipschitz velocity profiles and \vmp and \vmq be their associated modified velocities. Then
\vspace{-5mm}
\begin{align*}
\abs{\vmp(s, t) - \vmq(s, t)} \leq \norm{p - q} \qquad  \text{for all $s, t$.}
\end{align*}
\end{lem}
\Ss
\begin{proof}
For a fixed $s \in [0, t]$ and $t \in [0, T]$ suppose  
\begin{align*}
  \vmp(s, t) = \vint{p}_{\tau_1, t}
\qquad \text{and} \qquad
  \vmq(s, t) = \vint{q}_{\tau_2, t} .
\end{align*}
Without loss of generality, assume that $\vmp(s, t) \geq \vmq(s, t)$. Then
\begin{gather*}
\abs{\vmp(s, t) - \vmq(s, t)}
= \vint{p}_{\tau_1, t} - \vint{q}_{\tau_2, t}
\leq \vint{p}_{\tau_2, t} - \vint{q}_{\tau_2, t}
\leq \norm{p - q} .   \qedhere
\end{gather*}
\end{proof}

%\Ss

\begin{prop} \label{prop:Lip-f-rec}
Let $\big\{ p, \Eta^{(p)}, \{ f_n^{(p)}\}_{n=1}^\8, \frec^{(p)} \big\}$ and 
$\big\{ q, \Eta^{(q)}, \{ f_n^{(q)}\}_{n=1}^\8, \frec^{(q)} \big\}$ be two systems of disk-gas dynamics satisfying Assumptions
\ref{a0}-\ref{a2}. Then there exist a positive constant $Q_6$ that does not depend on $n$ such that the gas densities 
$ \{ f_n^{(p)} \}_{n=1}^\8$ and $\{ f_n^{(q)} \}_{n=1}^\8 $ satisfy the bound
\begin{align}
\abs{f_{n+1}^{(p)}(s, t) - f_{n+1}^{(q)}(s, t)} &\leq  
3^n Q_6 \left( \alpha_n(s) + \alpha_{n-1}(s) \right) \| p - q \| \,, \label{bound:Lip-f-n-p-q}
\qquad  n \geq 0 \,.
\end{align}
Consequently, for all $s \in [0, t]$ and $t \in [0, T]$ we have
\begin{align} \label{cond:Lip-f-rec-p-q}
 \abs{\frec^{(p)}(s, t) - \frec^{(q)}(s, t)} 
&\leq  4 Q_6 e^{6M^2 s} \| p - q \|. 
\end{align}

\end{prop}

\begin{proof}
We show the bounds in~\eqref{bound:Lip-f-n-p-q} by a similar induction proof as for Proposition~\ref{prop:bound-f-n-s}. First, the difference in density fluxes of $\phi_0$ satisfies
\begin{align*}
    \abs{j_0^{(p)}(t) - j_0^{(q)}(t)} 
&= \abs{\int_{-\8}^{ \vmp(0, t) } \left( p(t) - v \right) \phi_0(v) \dv - 
\int_{-\8}^{ \vmq(0, t) } \left( q(t) - v \right) \phi_0(v) \dv} 
\\%[2pt]
&= \abs{\int_{\vmq(0, t)}^{ \vmp(0, s) } \left( p(t) - v \right) \phi_0(v) \dv - 
\int_{-\8}^{ \vmq(0, t) } \big[ \left( p(t) - v \right) - \left( q(t) - v \right) \big] \phi_0(v) \dv} 
\\[3pt]
&\leq \| \vmq - \vmp \| \| \phi_0 \|_{\8} + \| p - q \| \| \phi \|_1 
= \left( \| \phi_0 \|_{\8} + \| \phi_0 \|_1 \right) \| p - q \| \,. 
\end{align*}
Therefore, we have
\begin{align*}
\left| f_1^{(p)}(s, t) - f_1^{(q)}(s, t) \right| &= 
\left| 2e^{ -( v_p(s, t) - p(s) )^2 } j_0^{(p)}(s) - 2e^{ -( v_q(s, t) - q(s) )^2 } j_0^{(q)}(s)\right| \\[3pt]
&\leq 2\left| e^{ -( v_p(s, t) - p(s) )^2 } - e^{ -( v_p(s, t) - p(s) )^2 } \right| \abs{j_0^{(p)}(s)}
+ 2 \left| j_0^{(p)}(s) - j_0^{(q)}(s) \right| \\[4pt]
&\leq Q_1 \| p - q \| + 2\| p - q \|\left( \| \phi_0 \|_{\8} + \| \phi_0 \|_1 \right) 
\\[3pt]
& = \| p - q \| \left( Q_1 + 2\| \phi_0 \|_{\8} + 2\| \phi_0 \|_1 \right) \,.
\end{align*}
Thus by choosing $Q_4 = Q_1 + 2\| \phi_0 \|_{\8} + 2\| \phi_0 \|_1$ we complete the proof for $n=1$. For $n\geq 1$ we have
\begin{align*}
f_{n+1}^{(p)}(s, t) - f_{n+1}^{(q)}(s, t) 
&= 2e^{-( v(s, t) - p(s))^2 }j_n^{(p)}(s) - 2e^{-( v(s, t) - p(s))^2 }j_n^{(p)}(s) \\[3pt]
&\leq 2\left| e^{ -( v_p(s, t) - p(s) )^2 } - e^{ -( v_p(s, t) - p(s) )^2 } \right| j_n^{(p)}(s)
+ 2 \left| j_n^{(p)}(s) - j_n^{(q)}(s) \right| \\[3pt]
&\leq \| p - q \| \alpha_{n}(s) + 2 \left| j_n^{(p)}(s) - j_n^{(q)}(s) \right| . 
\end{align*}
We re-formulate the difference in density fluxes as
\begin{align*}
j_n^{(p)}(t) - j_n^{(q)}(t) &= \int_0^t \dds{\vmp(s, t)} (p(t) - v_p(s, t)) f_n^{(p)}(s, t) \ds
                             - \int_0^t \dds{\vmq(s, t)} (q(t) - v_q(s, t)) f_n^{(q)}(s, t) \ds \\[3pt]
&= \int_0^t \left[ \dds{\vmp(s, t)} - \dds{\vmq(s, t)} \right] (p(t) - v_p(s, t)) f_n^{(p)}(s, t) \ds \\[3pt]
&\quad \, - \int_0^t \dds{\vmq(s, t)} \Big[ (p(t) - v_p(s, t)) - (q(t) - v_q(s, t)) \Big] f_n^{(p)}(s, t) \ds \\[3pt]  
&\quad \, - \int_0^t \dds{\vmq(s, t)} ( q(t) - v_q(s, t) ) \Big[f_n^{(p)}(s, t) - f_n^{(q)}(s, t) \Big] \ds %\\[5pt] 
=: J_1 + J_2 + J_3.
\end{align*}
By integration by parts and Proposition \ref{prop:bound-f-n-s}, we write $J_1$ as
\begin{align}
J_1 &= \int_0^t \dds{} \Big[ \vmp(s, t) - \vmq(s, t) \Big] (p(t) - v_p(s, t)) f_n^{(p)}(s, t) \ds \label{bnd:J1}\\[3pt]
&= \left[ \Big( \vmp(s, t) - \vmq(s, t) \Big) (p(t) - v_p(s, t)) f_n^{(p)}(s, t) \right]_{s=0}^{s=t} \nn \\[3pt]
&\quad \, + \int_0^t \Big[ \vmp(s, t) - \vmq(s, t) \Big] \dds{v_p(s, t)} f_n^{(p)}(s, t) \ds \nn \\[3pt]
&\quad \, - \int_0^t \Big[ \vmp(s, t) - \vmq(s, t) \Big] (p(t) - v_p(s, t)) \dds{ f_n^{(p)}(s, t) } \ds %\\[10pt]
=: J_1^1 + J_1^2 - J_1^3. \nn
\end{align}
The boundary terms $J_1^1$ are only nonzero for $n=1$, so we write
\begin{align*}
\abs{J_1^1} 
&= \abs{\left[ \Big( \vmp(s, t) - \vmq(s, t) \Big) (p(t) - v_p(s, t)) f_n^{(p)}(s, t) \right]_{s=0}^{s=t}} 
\\[5pt]
&= \abs{\Big( \vmq(0, t) - \vmp(0, t) \Big) (p(t) - v_p(0, t)) f_n^{(p)}(0, t)} 
\leq 
\| p - q \|2MQ_1 \delta_{1n}.
\end{align*}
By Lemma~\ref{le:vmprops}, the second term $J_1^2$ satisfies
\begin{align*}
\abs{J_1^2} &= \abs{\int_0^t \Big[ \vmp(s, t) - \vmq(s, t) \Big] \dds{v_p(s, t)} f_n^{(p)}(s, t) \ds} %\\[5pt]
\leq \| p - q \| \frac{M}{2} \int_0^t Q_1\alpha_{n-1}(s) \ds = \| p - q \| \frac{Q_1}{4M} \alpha_n (t).
\end{align*}
Similarly, the third term $J_1^3$ is bounded as
\begin{align*}
\abs{J_1^3} 
&= \abs{\int_0^t \Big[ \vmp(s, t) - \vmq(s, t) \Big] (p(t) - v_p(s, t)) \dds{ f_n^{(p)}(s, t) } \ds} \\[5pt]
&\leq 2M\| p - q \| \int_0^t 3^{n-1} Q_3 \,  \vpran{\alpha_{n-1}(s) + \alpha_{n-2}(s)} \ds \\[5pt]
&= \| p - q \|\frac{ 3^{n-1} Q_3}{M} \vpran{\alpha_n (t) + \alpha_{n-1}(t)}.
\end{align*}
Combining estimates for $J_1^1$, $J_1^2$ and $J_1^3$ we get
\begin{align*}
\abs{J_1} \leq \frac{\| p-q \|}{M}
\left( \frac{Q_1}{4} \alpha_n (t) + 3^{n-1} Q_3 \vpran{\alpha_n (t) + \alpha_{n-1}(t)} + 2M^2Q_1 \delta_{1n} \right).
\end{align*}
The second term $J_{2}$ is bounded as
\begin{align*}
\abs{J_2} 
&= \abs{\int_0^t \dds{\vmq(s, t)} \Big[ (p(t) - v_p(, t)) - (q(t) - v_q(s, t)) \Big] f_n^{(p)}(s, t) \ds} \\[5pt] 
&\leq \| p - q \| M \int_0^t Q_1 \alpha_{n-1}(s) \ds =  \| p - q \| \frac{Q_1}{2M} \alpha_{n}(t).
\end{align*}
Using the induction assumption, we derive the bound for $J_3$ as
\begin{align*}
\abs{J_3} &= \abs{\int_0^t \dds{\vmq(s, t)} ( q(t) - v_q(s, t) ) \Big[f_n^{(p)}(s, t) - f_n^{(q)}(s, t) \Big] \ds} \\[5pt] 
&\leq \| p - q \| M^2 \int_0^t 3^{n-1} Q_5 \left( \alpha_{n-1}(s) + \alpha_{n-2}(s) \right) \ds %\\[5pt]
\\[2pt]
&= \| p - q \| \frac{3^{n-1} Q_5}{2}\left( \alpha_n(t) + \alpha_{n-1}(t) \right) \,.
\end{align*}
Let $Q_5 = \frac{3Q_1}{2M} + 1 + \frac{2Q_3}{M} + 4MQ_1$. Then combining the estimates on $J_1$, $J_2$ and $J_3$ gives
\begin{align*}
\left| f_{n+1}^{(p)}(s, t) - f_{n+1}^{(q)}(s, t) \right| 
&\leq \| p - q \| \alpha_{n}(s) + 2 \left( J_1 + J_2 + J_3 \right) \\[5pt]
&\leq \| p - q \| \left( 2 \cdot 3^{n-1}Q_5 \big( \alpha_n(s) + \alpha_{n-1}(s) \big) + Q_5 \delta_{1n}\right);
\end{align*}
for $n=1$ it becomes
\begin{align*}
\left| f_{2}^{(p)}(s, t) - f_{2}^{(q)}(s, t) \right| 
&\leq \| p - q \| 
  \left(  2Q_5 \big( \alpha_1(s) + \alpha_0(s) \big) + Q_8 \right) 
\leq 
  3\| p - q \|Q_5 \big( \alpha_1(s) + \alpha_0(s) \big),
\intertext{while for $n \geq 2$ we get}
\left| f_{n+1}^{(p)}(s, t) - f_{n+1}^{(q)}(s, t) \right| &\leq \| p - q \| 
\left(  2 \cdot 3^{n-1}Q_5 \big( \alpha_n(s) + \alpha_{n-1}(s) \big) \right)
\leq \| p - q \| 3^nQ_5 \big( \alpha_n(s) + \alpha_{n-1}(s) \big) .
\end{align*}
Choosing $Q_6 = \max\{ Q_4, Q_5\}$ unifies all cases. 
Estimate~\eqref{cond:Lip-f-rec-p-q} follows by summing the bounds in~\eqref{bound:Lip-f-n-p-q}:
\begin{align*}
\abs{\frec^{(p)}(t) - \frec^{(q)}(t)} &\leq \sum_{n=1}^\8 \abs{f_n^{(p)}(t) - f_n^{(q)}(t)}
\\
&\leq \vpran{Q_6\sum_{n = 0}^\8 3^n \vpran{\alpha_n(s) + \alpha_{n-1}(s)}} \norm{p - q}
\\
&= 4 Q_6 e^{6M^2 s} \norm{p - q} .  \qedhere
\end{align*}  
\end{proof}

The Lipschitz property of the drag force is an immediate consequence of Proposition~\ref{prop:Lip-f-rec}:
\begin{prop} \label{prop:Lip-G}
Given two disk velocity profiles $p$ and $q$, let $G_p$ and $G_q$ be the corresponding drag forces defined in~\eqref{def:drag}. 
Then for any $T > 0$ there exists a constant $L_T$ such that
\begin{align*}
\norm{G_p - G_q} \leq L_T \norm{p - q} . 
\end{align*}
\end{prop}

\begin{proof}
Recall that we decompose $G_p = G_{p, 0} + G_{p, rec}$ and $G_q = G_{q, 0} + G_{q, rec}$. The Lipschitz bounds for 
$\abs{G_{p, 0} - G_{q, 0}}$ and $\abs{F(\eta_p(t), t) - F(\eta_q(t), t)}$ can be derived by direct estimates, so we focus on 
the re-collision part. We only give a sketch of the proof since it is very similar (and at times easier) to the one
for~\eqref{bound:Lip-f-n-p-q}. To simplify the notation we let \vspace{-5mm}
\begin{align*}
A_p(s, t) = (p(t)-v_p(s, t))^2 + \frac{\sqrt{\pi}}{2} (p(t) - v_p(s, t)).
\end{align*}
Then the difference becomes
\begin{align*}
G_{p, \text{rec}}(t) - G_{q, \text{rec}}(t) 
&= \int_0^t \dds{\vmp(s, t)} 
      A_p(s, t) \, \frec^{(p)}(s, t) \ds 
      -\int_0^t \dds{\vmq(s, t)} A_q(s, t) \frec^{(q)}(s, t) \ds 
\\[2pt]
& = \int_0^t \vpran{\dds{\vmp(s, t)} - \dds{\vmq(s, t)}} A_p(s, t)\frec^{(p)}(s, t) \ds
\\[2pt]
& \quad \, 
  + \int_0^t \dds{\vmq(s, t)} \vpran{A_p(s, t) - A_q(s, t)}\frec^{(p)}(s, t) \ds
\\[2pt]
& \quad \,
  + \int_0^t \dds{\vmq(s, t)} A_q(s, t)
      \vpran{\frec^{(p)}(s, t) - \frec^{(q)(s, t)}} \ds
= : K_1 + K_2 + K_3 \,.
\end{align*}
A Lipschitz bound for $K_1$ follows from the same calculation as \eqref{bnd:J1} in the proof of proposition 
\ref{prop:Lip-f-rec}. A Lipschitz bound for $K_2$ follows from the definition of $v_p$, and a Lipschitz bound for $K_3$ 
follows from~\eqref{cond:Lip-f-rec-p-q} together with the bounds for 
$\del_s \vmq$ in Lemma~\ref{le:vmprops}.  The combination of the three bounds for $K_1, K_2, K_3$ gives the Lipschitz bound for
$G_{\text{rec}}$.
\end{proof}

\begin{rmk}
Strictly speaking, the drag force in Proposition~\ref{prop:Lip-G} is only the contribution from the right-side of the disk.
However, as mentioned earlier, a similar Lipschitz property holds for the full drag force defined in~\eqref{def:drag-full},
since the estimates for the left side follow from a similar argument. The main modification needed for the left side is 
to change the definition of the modified average velocity $\vm$ into 
\begin{align*}
   \bar v(s, t) = \max_{\tau \in [s, t]} \vint{p}_{s, t} .
\end{align*}
Since replacing minimum with maximum does not affect the properties of the modified average velocity in Lemma~\ref{le:vmprops},
the rest of the estimates remain the same. 
\end{rmk}

%\newpage
We now have all the ingredients to prove the main result of this paper:
\begin{proof}[Proof of Theorem~\ref{thm:main}]
For any $T > 0$, by Proposition~\ref{prop:Lip-G} and the assumption that the external force $F(\cdot, t)$ is Lipschitz, we have
\vspace{-5mm}
\begin{align*}
\| p - q \|_{L^\infty(0, t)} \leq \vpran{L_T + \text{Lip}(F)} \int_0^T \| p - q \|_{L^\infty(0, t)} \dt ,
\end{align*}
which gives $p = q$ on $[0, T]$ by Gronwall's inequality. Meanwhile, for a given $p$, the density function
$f$ on the disk can be written explicitly using the decomposition established in Section \ref{sec:gas_decomp}: 
\begin{align*}
   f_R(\eta(t), v, t) = \sum_{n=0}^\infty f_{R, n}(\eta(t), v, t) ,
\qquad
   f_L(\eta(t), v, t) = \sum_{n=0}^\infty f_{L, n}(\eta(t), v, t),
\end{align*}
together with the initial condition $f(x, v, 0) = \phi_0(v)$. Therefore, the boundary conditions in~\eqref{BC:R}-\eqref{BC:L} are uniquely defined, which combined with the free transport equation~\eqref{eq:gas} gives a unique solution for $f$. We thus obtain a unique solution to the full gas-disk system. 
%so $f^{(p)} = f^{(q)}$ if $p=q$ and $f^{(p)}(x, v, 0) = f^{(q)}(x, v, 0) = \phi_0(v)$, .
\end{proof}

%\blue{Conclusion?..}

\Ss
\noindent {\bf Acknowledgements:}
The authors want to thank Ralf Wittenberg for fruitful discussions on this problem and pointing out a mistake in an earlier version. The research of W.S. is supported by NSERC Discovery Individual Grant R611626.

%%%%%%%%%%%%%%%%%%%%%%%%%%%%%%%%%%%%%%

%\vspace{15mm}

\bibliographystyle{dinat}
\bibliography{../Global/GasDisk}
%%%%%%%%%%%%%%%%%%%%%%%%%%%%%%%%%%%%%%

%\newpage

\end{document}